\documentclass{article}

\usepackage{amsmath}
\usepackage{amssymb, mathrsfs}
\usepackage{amsthm}
\usepackage[toc,page]{appendix}

\makeatletter
\newcommand{\customlabel}[2]{%
\protected@write \@auxout {}{\string \newlabel {#1}{{#2}{}}}}
\makeatother

\newcommand{\R}{\mathbb{R}}

\newcommand{\N}{\mathbb{N}}
\newcommand{\Prob}[1]{\mathcal{P}(#1)}
\newcommand{\C}{\mathcal{C}}
\newcommand{\Meas}[1]{\mathcal{M}(#1)}
\newcommand{\id}{\mathrm{Id}}

\newcommand{\Proj}[1]{{\mathrm{Proj}_{#1}}}

\newcommand{\xbar}{\overline{x}}
\newcommand{\ybar}{\overline{y}}
\newcommand{\zbar}{\overline{z}}
\newcommand{\ubar}{\overline{u}}
\newcommand{\vbar}{\overline{v}}
\newcommand{\spt}[1]{\mathrm{spt}(#1)}

\newtheorem{Thm}{Theorem}[section]
\newtheorem*{Thm*}{Theorem}
\newtheorem{Lem}[Thm]{Lemma}
\newtheorem{Prop}[Thm]{Proposition}
\newtheorem{Cor}[Thm]{Corollary}

\theoremstyle{definition}
\newtheorem{Def}[Thm]{Definition}
\newtheorem{Prbm}{Problem}
\newtheorem*{Prbm*}{Problem}
\newtheorem{Ex}{Example}

\theoremstyle{remark}
\newtheorem{Rmk}[Thm]{Remark}

\newcommand\numberthis{\addtocounter{equation}{1}\tag{\theequation}}

\title{Quadratic optimal transportation problem with a positive semi definite structure on the cost function}
\author{Seonghyeon Jeong}
\date{}

\begin{document}
\maketitle

\section{Introduction}\label{sec: intro}

\subsection{Introduction}
Optimal Transportation problem (OT) has been an active research area in last decades, and many questions in the field are studied. Furthermore as OT gives a way to metrize the probability spaces, OT has a wide range of applications. While OT was used in wide application area, there were some application situations where a few aspects of OT needed to be modified to fit in the situation. As such, there has been a lot of versions of OT such as multi-marginal OT (\cite{pass2015multi}, \cite{kim2015multi}, \cite{pass2022monge}, \cite{Gero2019simple}), optimal partial transportation problem(\cite{Alessio2010partial}, \cite{Kitagawa2015multi}, \cite{Davila2016dynamics}), congested transportation problem(\cite{Carlier2008congestion}), martingale transportation problem(\cite{Beiglbock2016martingale}, \cite{Ghoussoub2019martingale}, \cite{beiglbock2017martingale}), co-OT(\cite{redko2020coot}), etc. 
In this paper, we introduce another version of OT.
\begin{Prbm*}[0-1]\customlabel{Prbm 0-1}{(0-1)}
\begin{equation}\label{eqn: form}
\inf_{T\sharp \mu = \nu} \iint c(x, T(x), y, T(y) ) d\mu d\mu.
\end{equation}
\end{Prbm*}
Problem \ref{Prbm 0-1} can be regarded as a quadratic version of OT, written in Monge's formulation. The problem that corresponds Katorovich's relaxed formation of OT is the following.
\begin{Prbm*}[0-2]\customlabel{Prbm 0-2}{(0-2)}
\begin{equation}
\inf_{\pi \in \Pi(\mu,\nu)} \iint c d\pi d\pi.
\end{equation}
\end{Prbm*}
Some readers may familiar with this problem already due to the studies about Gromov-Wasserstein distance.
\begin{equation*}
\inf_{\pi \in \Pi(\mu,\nu)} \iint \left| \| x-\xbar \|^2 - \| y- \ybar \|^2 \right|^2 d\pi d\pi.
\end{equation*}
This type of quantity can also be found in \cite{sturm2020space}, where the author defines $L_p$ distortion distance between two metric measure spaces in the form of \eqref{eqn: form} and construct the space of metric measure spaces. We would like to comment that there is another result \cite{cabrera2022optimal} by R. Cabrera about where the author studies Problem \ref{Prbm 0-1} with dynamic setting. 

The motivation of Problem \ref{Prbm 0-1} in this paper, however, arises from a data dimension reduction algorithms such as t-SNE \cite{JMLR:v9:vandermaaten08a} and UMAP \cite{UMAP}. In data science, it is an important problem to design an algorithm which let us save the data points from very high dimension in low dimension such as $\R^2$ or $\R^3$ preserving some structures of the point clouds. For instance, in t-SNE algorithm which is designed in \cite{JMLR:v9:vandermaaten08a}, the authors define affinities $p_{ij}$ and $q_{ij}$ which quantifies similarity of two points $x_i$ and $x_j$ in high dimension and $y_i$ and $y_j$ in low dimension respectively, and use KL-divergence $\sum p \log \frac{p}{q}$ to measure the difference between $p_{ij}$ and $q_{ij}$. Then t-SNE algorithm seeks a set of points in the low dimension which minimizes the KL-divergence. Existence of such set of points $\{ y_i \}$ in the low dimension was studied in \cite{auffinger2023equilibrium}, and \cite{jeong2024convergence}. Then, letting $T$ be a function such that $T(x_i) = y_i$, we see that $T$ solves Problem \ref{Prbm 0-1} with $c = p\log\frac{p}{q}$. Hence, to study the structure of the projection $T$, it is necessary to study Problem \ref{Prbm 0-1}.

The goal of this paper is to introduce the quadratic optimal transportation problem . We explore some basic properties of the quadratic optimal transportation problem analogous to OT, and discuss extra structures of the quadratic optimal transportation problem. The main theorem of this paper is the duality for the quadratic transportation problem with a structural assumption on the cost function. 
\begin{Thm*}[Duality, rough statement]
For squared cost function $c$, we have the following identity
\begin{equation*}
\inf_{\pi} \sqrt{\iint c d\pi d\pi} = \sup_{(f,\phi,\psi)} \left(\int \phi d\mu + \int \psi d\nu \right)
\end{equation*}
where the infimum and supremum are taken over some admissible sets.
\end{Thm*}
The squared cost function is defined in Section \ref{sec: prop} and the admissible sets are described in Section \ref{sec: qOT} and \ref{sec: dual}.

This paper consists as follows. In section \ref{sec: qOT}, we introduce the quadratic optimal transportation problem and compare the problem with other versions of OT. In section \ref{sec: prop}, we explore some basic properties of the quadratic optimal transportation problem. In \ref{sec: dual}, we provide a duality formula for a certain type of cost function. In the appendices, provide some background theory for the classical optimal transportation problem which are used in this paper, we give a proof of Lemma \ref{Lem: OT ex lem}, and we formally derive a related partial differential equation.

\subsection{Notations}
We introduce some notations which are used in this paper here. $X$ and $Y$ are compact Polish spaces, and $Z = X \times Y$. We will use $x$ and $\xbar$ ($y$ and $\ybar$) to denote the points in $X$ (in $Y$). We also use $z = (x,y)$ and $\zbar = (\xbar,\ybar)$ to denote the points in $Z$. $\Meas{X}, \Meas{Y}, \Meas{Z}$ are the sets of Borel measures on the respective spaces, and $\Prob{X}, \Prob{Y}, \Prob{Z}$ are the sets of probability measures on the respective spaces. $\mu$ and $\nu$ will denote probability measures on $X$ and $Y$ respectively. $B_r^d(p)$ will denote the balls of radius $r$ centered at $p$ in $\R^d$. When we use a ball in a space $X$ that is not a Euclidean space, we use $B^X_r(p)$ to denote the ball. $c$ will be the cost function, which often has domain $Z \times Z$, but the domain of $c$ can be changed accordingly. $\Proj{\cdot}$ is the canonical projection to the space indicated in the subscription.

\section{Quadratic transportation problem}\label{sec: qOT}

\subsection{Quadratic transportation problem}
We restate the problem that we consider in this paper in detail. Let $c: Z\times Z \to \R$ be a lower semi-continuous function. Then the problem states as follows. 

\begin{Prbm} \label{prbm: qOT monge}
Find a measurable function $T : X \to Y$ which achieves the following infimum
\begin{equation*}
\inf \iint c(x,T(x), \xbar, T(\xbar))d\mu(x)d\mu(\xbar)
\end{equation*} 
among the measurable functions that satisfies $T_\sharp \mu = \nu$, where $T_\sharp \mu$ is the push-forward measure.
\end{Prbm}
We call the above problem \emph{Monge formulation of the quadratic optimal transportation problem}, or simply \emph{Monge QOT}. We call a solution to Monge QOT a \emph{Monge solution}. Monge QOT is highly non-linear with respect to $T$. However, Using Kantorovich's relaxation as in the classical optimal transportation theory, we can reduce the non-linearity to the quadratic dependency. We first define
\begin{equation*}
\Pi(\mu,\nu) = \{ \pi \in \Prob{Z}| \Proj{X}_\sharp \pi = \mu, \Proj{Y}_\sharp \pi = \nu \}.
\end{equation*}
Then the \emph{Kantorovich formulation of the quadratic optimal transportation problem (Kantorovich QOT)} is the following.

\begin{Prbm}\label{prbm: qOT}
Find $\pi \in \Pi(\mu,\nu)$ that achieves the following infimum.
\begin{equation*}
\inf_{\pi \in \Pi(\mu,\nu)} \iint c(x,y,\xbar,\ybar) d\pi(x,y)d\pi(\xbar,\ybar). 
\end{equation*}
\end{Prbm}
We call a solution to the Kantorovich QOT a \emph{Kantorovich solution}. We will often use quadratic optimal transportation problem or QOT to refer one of Monge QOT or Kantorovich QOT, and we will often say solution when there is no need for distinguishing the Monge solutions and the Kantorovich solutions.

The Kantorovich formulation of the classical optimal transportation problem allow us to view the problem in the perspective of linear programming. In the quadratic optimal transportation case, the relaxed problem \ref{prbm: qOT} can be considered as a quadratic programming. 

\subsection{Comparison with other variations of OT}
In this subsection, we claim that the solutions of the quadratic optimal transportation are different from the solutions to other versions of the optimal transportation problem with naive modifications.

\subsubsection{Classical optimal transportation problems}
We start with the classical optimal transportation problems. In this problem, we consider a cost function $c:X \times Y \to \R$ and measures $\pi$ which have given marginals $\mu \in \Prob{X}$ and $\nu \in \Prob{Y}$. We define total cost $\C : \Prob{X \times Y} \to \R \cup \{ \pm \infty \} $ by $\C(\pi) = \int c d \pi$. Then we seek a probability measure in $\Pi(\mu,\nu)$ which minimize the total cost.

\begin{Prbm}\label{prbm: OT}
\begin{equation*}
\inf_{\pi \in \Pi(\mu,\nu)} \C(\pi) = \inf_{\pi \in \Pi(\mu,\nu)} \int c d\pi.
\end{equation*}
\end{Prbm}

In the quadratic optimal transportation problem, we consider probability measures with the same marginal conditions in $\Pi(\mu,\nu)$, but we integrate the cost function against the measure twice. Hence, one may consider Problem \ref{prbm: OT} with $\pi \otimes \pi$ as a measure in $\Prob{X^2 \times Y^2}$ which is in $\Pi(\mu \otimes \mu, \nu \otimes \nu)$. 

\begin{Prbm*}{\textbf{3-1}}\customlabel{prbm: OT2}{3-1}
\begin{equation*}
\inf_{\pi' \in \Pi(\mu\otimes\mu, \nu\otimes\nu)} \C(\pi') = \inf_{\pi' \in \Pi(\mu\otimes\mu, \nu\otimes\nu)} \int c d\pi'.
\end{equation*}
\end{Prbm*}

We claim that the solution to this problem is different from the solutions to the QOT in general with an example. Let us describe the idea before we give an example. The idea is to see the structure of the solutions to the two problems. If QOT has a Monge solution $T$, then $\pi = (Id \times T)_\sharp \mu$ and $\pi \otimes \pi = (Id\times T \times Id \times T)_\sharp \mu$. In particular, if $\pi\otimes \pi$ is a Kantorovich solution of Problem \ref{prbm: OT2}, then $(T,T)$ will be the Monge solution to the Problem \ref{prbm: OT2}. Then the horizontal and vertical slices of $X\times X$ will be mapped to the horizontal and vertical slices of $Y \times Y$ respectively. This is not true for the solutions to the classical optimal transportation problem in general. 

\begin{Ex}\label{Ex: OT}
Let $X = Y = [0,1]$ and $c(x,y,\xbar,\ybar) = -(x+\xbar)(y+\ybar)$. We let $\mu$ and $\nu$ be the probability measures on $[0,1]$ defined by
\begin{align*}
\mu &= \frac{2}{3} dx\lfloor_{[0,\frac{1}{2}]} + \frac{4}{3} dx\lfloor_{[\frac{1}{2},1]}, \\
\nu &= dy\lfloor_{[0,1]}.
\end{align*}
Let $\pi'$ be a solution to the Problem \ref{prbm: OT2} with $c, \mu, \nu$ just defined. 

\begin{Lem}\label{Lem: OT ex lem}
Define $\overline{D}_k$ and $\underline{D}_k$ by 
\begin{align*}
\overline{D}_k & \{ (x, \xbar) | x + \xbar \leq k \}, \\
\underline{D}_k & \{ (x, \xbar) | x + \xbar \geq k \}.
\end{align*}
Let $(x_0, \xbar_0 , y_0, \ybar_0) \in \spt{\pi'}$ and $k = x_0 + \xbar_0$ and $l = y_0 + \ybar_0$. Then
\begin{align*}
\int_{\overline{D}_k} d\mu \otimes \mu = \int_{\overline{D}_l} d\nu \otimes \nu, \\
\int_{\underline{D}_k} d\mu \otimes \mu = \int_{\underline{D}_l} d\nu \otimes \nu.
\end{align*}
\end{Lem}
We postpone the proof of Lemma \ref{Lem: OT ex lem} to the Appendix \ref{sec: proof lem}.

\begin{Cor}\label{Cor: OT no QOT}
\begin{equation*}
\inf_{\gamma \in \Pi(\mu\otimes \mu, \nu \otimes \nu)} \int c d\gamma < \inf_{\pi \in \Pi(\mu,\nu)} \iint c d\pi d\pi.
\end{equation*}
\end{Cor}
\begin{proof}
we first note that 
\begin{equation*}
\inf_{\gamma \in \Pi(\mu\otimes \mu, \nu \otimes \nu)} \int c d\gamma \leq \inf_{\pi \in \Pi(\mu,\nu)} \iint c d\pi d\pi
\end{equation*}
as $\pi \otimes \pi \in \Pi(\mu \otimes \mu, \nu \otimes \nu)$ for any $\pi \in \Pi(\mu,\nu)$. Therefore, we only need to show that the equality is not attained. Suppose we have the equality
\begin{equation}\label{eqn: OT=QOT}
\inf_{\gamma \in \Pi(\mu\otimes \mu, \nu \otimes \nu)} \int c d\gamma = \inf_{\pi \in \Pi(\mu,\nu)} \iint c d\pi d\pi.
\end{equation}
Note that by existence of the solution of QOT (Proposition \ref{prop: exist opt plan}), $\exists \pi^* \in \Pi(\mu,\nu)$ which is attains the infimum. Then $\pi^*\otimes \pi^*$ is also a solution to Problem \ref{prbm: OT2}. Let $(x,y) \in \spt{\pi^*}$, then $(x,y,x,y) \in \spt{ \pi^* \otimes \pi^*}$ and by Lemma \ref{Lem: OT ex lem}, we observe
\begin{align*}
\int_{\overline{D}_{2x}} d\mu \otimes \mu = \int_{\overline{D}_{2y}} d \nu \otimes \nu, \\
\int_{\underline{D}_{2x}} d\mu \otimes \mu = \int_{\underline{D}_{2y}} d \nu \otimes \nu.
\end{align*}
Explicit computation with the definitions of $\mu$ and $\nu$, one can observe that $y$ is given by a function of $x$. In particular, if $y= T(x)$, one can compute
\begin{equation*}
T(x) = \frac{\sqrt2}{3}x \textrm{ for } 0 \leq x \leq \frac{1}{2}.
\end{equation*}
On the other hand, we also have $\nu = T_\sharp \mu$ as $\spt{\pi^*}$ is on the graph of $T$. Then we can also compute
\begin{equation*}
dy\lfloor_{[0,\frac{\sqrt2}{6}]}=\nu \lfloor_{[0,\frac{\sqrt2}{6}]} = T_\sharp \mu\lfloor_{[0,\frac{1}{2}]} = \left( x \mapsto \frac{\sqrt2}{3}x \right)_\sharp \frac{2}{3} dx = \frac{1}{\sqrt2} dy,
\end{equation*}
which is a contradiction. Therefore the equality \eqref{eqn: OT=QOT} cannot be attained.
\end{proof}

Corollary \ref{Cor: OT no QOT} shows that the solution to the classical optimal transportation problem \ref{prbm: OT} cannot be the solution to the QOT of this example.
\end{Ex}

\begin{Rmk}
By solving the formula of the cost function in Example \ref{Ex: OT}, one can find that the quadratic optimal transportation problem in Example \ref{Ex: OT} is in fact equivalent to a classical optimal transportation problem with $c(x,y) = -xy$. 
\begin{equation*}
\inf_{\Pi(\mu,\nu)}\iint -(x+ \xbar)(y+ \ybar)d\pi d\pi = 2\inf_{\Pi(\mu,\nu)} \left( \int - xy d\pi  - \int x d\mu \int y d\nu \right).
\end{equation*}
This shows that for some special cases, the quadratic optimal transportation problem can still be regarded as a classical optimal transportation problem, but with appropriate adjustments.
\end{Rmk}

\subsubsection{Multi-marginal problems}
In multi-marginal optimal transport problems, we consider probability measures $\mu_i \in \Prob{X_i}$, where $X_i$ are Polish spaces, and a cost function $c: \prod_{i=1}^n X_i \to \R$, and we seek a probability measure $\gamma$ that minimizes the total cost $\int c d\gamma$ among all the probability measures $\gamma$ that satisfy the marginal conditions $\Proj{X_i}_\sharp \gamma = \mu_i $. 
\begin{Prbm}\label{prbm: mOT}
\begin{equation*}
\inf_{\gamma \in \Pi(\mu_i | 1 \leq i \leq n)} \int c d \gamma
\end{equation*}
where 
\begin{equation*}
\Pi(\mu_i | 1 \leq i \leq n) = \left\{ \gamma \in \Prob{\prod_{i=1}^n X_i} | \Proj{X_i}_\sharp \gamma = \mu_i\right\}.
\end{equation*}
\end{Prbm}
Then one can consider the multi-marginal problem with $n=4$, $X_1 = X_3= X$, $X_2=X_4 = Y$, $\mu_1 = \mu_3 = \mu$ and $\mu_2=\mu_4 = \nu$.
\begin{Prbm*}{\textbf{4-1}}\customlabel{prbm: MOT2}{4-1}
\begin{equation*}
\inf_{\gamma \in \Pi(\mu,\mu,\nu,\nu)} \int c d \gamma,
\end{equation*}
where 
\begin{equation*}
\Pi(\mu,\mu,\nu,\nu) = \left\{ \gamma \in \Prob{\prod_{i=1}^n X_i} \bigg| \begin{matrix}\Proj{X_1}_\sharp \gamma = \Proj{X_3}_\sharp \gamma = \mu, \\ 
\Proj{X_2}_\sharp \gamma = \Proj{X_4}_\sharp \gamma = \nu \  \end{matrix} \right\}.
\end{equation*}
\end{Prbm*}
 We claim that this modification of multi-marginal problem does not give a solution to QOT in general. To see this, we consider the dimension of the supports of the solutions. In \cite{pass2015multi},\cite{kim2015multi}, it is shown that under suitable condition, the solution of Problem \ref{prbm: mOT} is given by a function, i.e. there exists $F_i : X_1 \to X_i$, $2 \leq i \leq n$ such that 
\begin{equation*}
(Id \times F_2 \times \cdots \times F_n)_\sharp \mu_1
\end{equation*}
is a solution to Problem \ref{prbm: mOT}. Then if we consider $X_i = X$ for all $i$ where $\dim X = d$, we can observe that the dimension of the support of the solution will be $d$ as it lies on the graph of a function whose domain has dimension $d$. On the other hand, if a solution $\pi$ of QOT is given by a Monge solution $T$, then the support of the measure $\pi\otimes \pi$ lies in the graph of the function $(T,T)$ which has $2d$-dimensional domain. 

\begin{Ex}
Let $X_i=[0,1]$ for $1\leq i \leq 4$ and we consider 
\begin{equation*}
c(x_1,x_2,x_3,x_4) = \sum_{i < j} \frac{1}{2} \| x_i - x_j \|^2.
\end{equation*}
Let $\mu_i = dx\lfloor_{[0,1]}$ for $1 \leq i \leq 4$. Then it is obvious that the infimum value of Problem \ref{prbm: mOT} is $0$ (Consider $\gamma = (Id \times Id \times Id \times Id)_\sharp dx\lfloor_{[0,1]}$). On the other hand, QOT with $X = Y = [0,1]$, $\mu=\nu = dx\lfloor_{[0,1]}$ and $c(x,y,\xbar,\ybar)$ where $c$ is defined above does not have the infimum value 0. 
\begin{align*}
&\iint c d \pi d\pi \\
=& \iint \frac{1}{2} \left( | x- y |^2 + | x- \xbar |^2 + | x- \ybar|^2 + | y - \xbar|^2 + | y- \ybar|^2 + | \xbar - \ybar |^2 \right) d\pi d\pi \\
\geq &\iint \frac{1}{2} \left( | x - \xbar |^2 + | y - \ybar |^2 \right) d \pi d\pi \\
=& \int_0^1 \int_0^1 (x-\xbar)^2 dxd\xbar >0.
\end{align*}
Therefore, solutions of multi-marginal problem does not directly implies solutions of QOT as
\begin{equation*}
\inf_{\gamma \in \Pi(\mu,\nu,\mu,\nu)} \int c d\gamma < \inf_{\pi \in \Pi(\mu,\nu)} \iint c d\pi d\pi.
\end{equation*}
\end{Ex}

\subsubsection{Co-optimal transportation problems}
Co-optimal transportation problem was introduced in \cite{redko2020coot}. 
\begin{Prbm}\label{prbm: coOT}
\begin{equation*}
\inf_{\pi \in \Pi(\mu_1,\nu_1), \rho \in \Pi(\mu_2,\nu_2)} \iint c d\pi d\rho.
\end{equation*}
\end{Prbm}
This problem have the most similar look with Problem \ref{prbm: qOT} as we get Problem \ref{prbm: qOT} by setting $\pi = \rho$ in Problem \ref{prbm: coOT}. Still, if we consider the Problem \ref{prbm: coOT} with $\mu_1 = \mu_2 = \mu$ and $\nu_1 = \nu_2 = \nu$, Solutions to Problem \ref{prbm: coOT} does not directly give solutions of QOT. Here, the idea of the example is to look at the difference of the supports of $\pi \otimes \rho$ and $\pi \otimes \pi$. Support of $\pi \otimes \pi$ always contains some part of the diagonal of $(X\times Y)^2$, while $\pi \otimes \rho$ does not have to contain any part of the diagonal of $(X \times Y)^2$. Hence, we can build an example by letting the cost function to have very high value on the diagonal of $(X\times Y)^2$. 

\begin{Ex}
We give a discrete example. We let $X = Y = \{ 0,1\}$ and define the cost function $c$ by
\begin{equation*}
c(x,y,\xbar,\ybar) = \left\{ \begin{matrix} 1 & x=\xbar \textrm{ and } y=\ybar \\
 0 & \textrm{otherwise}\end{matrix} \right.
\end{equation*}
Let $\mu = \nu = \frac{1}{2} \delta_{0} + \frac{1}{2} \delta_{1}$. We first consider Problem \ref{prbm: coOT}. Like stated above, we construct $\pi$ and $\rho$ so that $\spt{\pi \otimes \rho}$ does not contain any part of the diagonal of $(X \times Y)^2$. We let 
\begin{align*}
\pi &= \frac{1}{2} \delta_{(0,0)} + \frac{1}{2} \delta_{(1,1)}, \\
\rho& = \frac{1}{2} \delta_{(0,1)} + \frac{1}{2} \delta_{(1,0)}.
\end{align*}
It is not hard to observe that $\pi, \rho \in \Pi(\mu,\nu)$ and 
\begin{equation*}
\iint c d\pi d\rho = 0.
\end{equation*}
Therefore, infimum of Problem \ref{prbm: coOT} is 0 in this example. On the other hand, for any $\gamma \in \Pi(\mu,\nu)$, we have
\begin{equation*}
\gamma = A \delta_{(0,0)} + B \delta_{(0,1)} + C \delta_{(1,0)} + D \delta_{(1,1)}
\end{equation*}
for some $A,B,C,D \geq 0$ such that
\begin{equation*}
A+B=C+D=A+C=B+D=\frac{1}{2}.
\end{equation*}
Then we can compute
\begin{align*}
\iint c d\gamma d\gamma & = A^2 + B^2 + C^2 + D^2 \\
& \geq 4 \times \left( \frac{A+B+C+D}{4} \right)^2 = \frac{1}{4}.
\end{align*}
Therefore, we obtain
\begin{equation*}
\inf_{\pi,\rho \in \Pi(\mu,\nu)} \iint c d\pi d\rho < \inf_{\pi \in \Pi(\mu,\nu)} \iint c d\pi d\pi.
\end{equation*}
\end{Ex}

\subsubsection{Other transportation problems}
There are other variations of the optimal transportation problems such as weak optimal transportation problem, martingale optimal transportation problem, optimal transportation with congestion. Among these variations, optimal transportation with congestion problem considers interaction that occurs along the transportation. This problem was considered by \cite{Carlier2008congestion}, but most of the study consider the equations describing the congestion. In \cite{cabrera2022optimal}, however, the author adds a quadratic transportation term to describe the interaction. 

\section{Elementary properties}\label{sec: prop}

\subsection{Existence, non-uniqueness, and non-locality}
Existence of a Kantorovich solution to the Kantorovich QOT can be proved similarly to the existence proof of classical optimal transportation problem. Before we prove the existence, we show a short lemma about a tensor product of weakly converging sequences of probability measures.

\begin{Lem}\label{lem: weakly converging tensor product}
Suppose $\{ \mu_k \}_{k=1}^\infty $ and $\{ \nu_k \}_{k=1}^\infty$ be sequences of probability measures on Polish spaces $X$ and $Y$ that converge weakly to $\mu_0$ and $\nu_0$ respectively. Then the sequence of probability measures $\{ \mu_k \otimes \nu_k \}_{k=1}^\infty$ weakly converges to $\mu_0 \otimes \nu_0$.
\end{Lem}
\begin{proof}
By Prokhorov's Theorem, $\{ \mu_k \}_{k=1}^\infty$ and $\{ \nu_k \}_{k=1}^\infty$ are tight. We claim that $\{ \mu_k \otimes \nu_k \}_{k=1}^\infty$ is also tight. Indeed, for any $\epsilon >0$, there exist compact sets $K \subset X$ and $L \subset Y$ such that
\begin{equation*}
\mu_k[X \setminus K] < \frac{\epsilon}{2}, \ \nu_k[Y \setminus L] < \frac{\epsilon}{2}
\end{equation*}
for any $k$. Then
\begin{equation*}
\mu_k \otimes \nu_k [(X \times Y) \setminus (K \times L)] \leq \mu_k[X \setminus K ] + \nu_k [Y \setminus L] < \frac{\epsilon}{2}
\end{equation*}
for any $k$. Noting that $K \times L$ is also compact, $\{ \mu_k \otimes \nu_k \}_{k=1}^\infty$ is also tight. Then, invoking Prokhorov again, any subsequence of $\{ \mu_k \otimes \nu_k \}_{k=1}^\infty$ has a subsequence $\{ \mu_{k_i} \otimes \nu_{k_i} \}_{i=1}^\infty$ which converges weakly to a probability measure $\gamma$. Then, we observe
\begin{align*}
\gamma[A \times B] & = \lim_{i \to \infty} \mu_{k_i} \otimes \nu_{k_i} [A \times B] \\
& = \lim_{i \to \infty} \mu_{k_i} [A] \times \nu_{k_i}[B]\\
& = \mu_0 [A] \times \nu_0[B].
\end{align*}
Therefore we obtain that $\gamma = \mu_0 \otimes \nu_0$. We just proved that any subsequence of $\{ \mu_k \otimes \nu_k\}_{k=1}^\infty$ has a subsequence that converges to $\mu_0 \otimes \nu_0$ weakly, and this implies that the whole sequence $\{ \mu_k \otimes \nu_k \}_{k=1}^\infty$ converges to $\mu_0 \otimes \nu_0$ weakly.
\end{proof}

\begin{Prop}\label{prop: exist opt plan}
Problem \ref{prbm: qOT} admits a solution.
\end{Prop}
\begin{proof}
Let $\{ \pi_k \}_{k=1}^\infty \subset \Pi(\mu,\nu)$ be a minimizing sequence.
\begin{equation*}
\lim_{k \to \infty} \iint c d\pi_k d\pi_k = \inf_{\pi \in \Pi(\mu,\nu)} \iint c d\pi d\pi.
\end{equation*}
Since $\Pi(\mu,\nu)$ is a tight set, $\pi_k$ converges to $\pi_0 \in \Pi(\mu,\nu)$ weakly up to a subsequence (which we still denote with $\pi_k$). Then by Lemma \ref{lem: weakly converging tensor product}, we obtain that $\pi_k \otimes \pi_k$ converges to $\pi_0 \otimes \pi_0$ weakly. Also, as $c$ is lower semi-continuous, there exists $\{ c_n \}_{n=1}^\infty$, a sequence of continuous functions such that $c_n \leq c$ and $c_n \nearrow c$ pointwisely. Then
\begin{align*}
\inf_{\pi \in \Pi(\mu,\nu)} \iint c d\pi d\pi & = \lim_{k \to \infty} \iint c d\pi_k d\pi_k \\
& \geq \lim_{k \to \infty} \iint c_n d \pi_k d\pi_k \\
& = \iint c_n d\pi_0 d\pi_0, 
\end{align*}
where we have used the weak convergence of $\pi_k \otimes \pi_k$ to $\pi_0 \otimes \pi_0$ in the last equality. We take $n \to \infty$ and invoke monotone convergence, we obtain
\begin{equation*}
\inf_{\pi \in \Pi(\mu,\nu)} \iint c d\pi d\pi \geq \iint c d\pi_0 d\pi_0.
\end{equation*}
This implies that $\pi_0$ is a minimizer.
\end{proof}

To show the existence of a Monge solution, one need to show that the Kantorovich solution $\pi$ is supported in the graph of a function $T:X \to Y$ which will be a Monge solution. In classical OT problem, this was studied with regularity of a potential function $\phi$, often using Monge-Amp\`ere type equations (\cite{Figalli2013holder}, \cite{Loeper2009reg}, \cite{MTW2005reg}). In the appendix of this paper, we derive a Monge-Amp\`ere type equation formally using a similar argument. In case of QOT, however, it seems that the Monge-Amp\`ere type equation that we obtain is non-local due to the non-local property as described in Example \ref{Ex: non local}. Although there are some discussions about non-local Monge-Amp\`ere equations (\cite{McCann2018unequal}, \cite{caffa2014nonloc}, \cite{Jhavaeri2020fractional}), The Monge-Amp\`ere type equation derived for QOT seemed to have a different form from the existing non-local Monge-Amp\`ere type equations.

Another elementary property that we can ask after the existence is the uniqueness of the solution. Like classical optimal transportation problem, QOT may have multiple solutions in general. 

\begin{Ex}\label{ex: qOT non unique}
Let $X=Y $  be $ \overline{B_1^2(0)}$, the unit disk in $\R^2$, and let $\mu= \nu = \frac{1}{\pi} d \mathcal{L}$, the uniform probability measure on the unit disk. We choose the cost function $c$ to be the $L_p$ distortion distance cost. 
\begin{equation*}
c(x,y,\xbar,\ybar) = \left| \|x-\xbar\|^2-\|y-\ybar\|^2 \right|^2.
\end{equation*} 
Then let $\pi_0 = (\id \times \id)_\sharp \mu \in \Pi(\mu,\nu)$. We compute
\begin{align*}
\iint c(x,y,\xbar,\ybar) d\pi_0 d\pi_0 & = \iint \left| \|x-\xbar\|^2-\|y-\ybar\|^2 \right|^2 d(\id \times \id)_\sharp \mu d (\id \times \id)_\sharp \mu \\
& = \iint \left| \| x - \xbar \|^2 - \| x - \xbar \|^2 \right|^2 d\mu d\mu \\
& =0.
\end{align*} 
Since $c \geq 0$, above computation shows that $\pi_0$ is a solution to QOT. On the other hand, let $R_\theta$ be the rotation by angle $\theta$. In particular, $R_\theta$ is an isometry. Let $\pi_0' = (\id \times R_\theta)_\sharp \mu \in \Pi(\mu,\nu)$. Then 
\begin{align*}
\iint c(x,y,\xbar,\ybar) d\pi_0 d\pi_0 & = \iint \left| \|x-\xbar\|^2-\|y-\ybar\|^2 \right|^2 d(\id \times R_\theta)_\sharp \mu d (\id \times R_\theta)_\sharp \mu \\
& = \iint \left| \| x - \xbar \|^2 - \| R_\theta x - R_\theta \xbar \|^2 \right|^2 d\mu d\mu \\
& = \iint \left| \| x - \xbar \|^2 - \| x - \xbar \|^2 \right|^2 d\mu d\mu \\
& =0.
\end{align*}
Hence $\pi_0'$ is also a solution of QOT. However, $\pi_0 \neq \pi_0'$ as their supports are different.
\end{Ex}

\begin{Rmk}
We remark here that in Example \ref{ex: qOT non unique}, the measures $\mu$ and $\nu$ were absolutely continuous with respect to the Lebesgue measure. In the classical optimal transportation problem, however, the solution would be unique when both $\mu$ and $\nu$ are absolutely continuous. The idea of Example \ref{ex: qOT non unique} is to use the symmetry of the cost function.
\end{Rmk}

The classical optimal transportation problems have locality property such as Proposition \ref{prop: ot local}. Roughly speaking, restriction of a solution is still optimal with corresponding marginals. In QOT case, however, it is not true in general due to the quadratic dependency which leads to non-local behavior of the solutions.

\begin{Ex}\label{Ex: non local}
Let $X=Y= \{ 0,1,2 \}$. We define a cost function $c$ as follows. We first let $c(x,y,\xbar,\ybar) = c(\xbar, \ybar, x, y)$, and define 
\begin{align*}
 c(1,0,1,0) = c(1,0,0,1) = c(0,1,0,1) & = 0, \\
 c(2,2,0,0) = c(2,2,1,1) = c(2,2,2,2) & = 0, \\
 c(2,2,1,0) = c(2,2,0,1) & = 10
\end{align*} 
Then we define the cost function to be 1 at the points where the value of $c$ is not assigned above. We consider $\mu = \nu = \frac{1}{3} \sum_{i=0}^2 \delta_i$. It is not hard to check that the solution is given by $T = \id$ with minimum cost $\frac{4}{3}$. next, we consider $\mu' = \nu' = \frac{1}{2} \sum_{i=0}^1 \delta_i$. Note that we have $\mu' =\nu' = T_\sharp \mu' \ll \mu = \nu$. However, the solution to the quadratic optimal transportation problem in this case is given by the map $T'$
\begin{align*}
T'(0) = 1, & \ T'(1) = 0,
\end{align*}
with minimum cost $0$. However, we have $T \neq T'$.
\end{Ex}

\subsection{Monotonicity and positive semi-definite costs}
In the classical optimal transportation problem, the Kantorovich solutions are $c$-cyclically monotone (\ref{def: cyclical monotone}) by Lemma \ref{lem: cyclical monotone OT solution}. We can expect a similar structure in QOT case. For instance, we let 
\begin{equation}\label{eqn: c_pi}
c_{\pi}(x,y) = \int c(x,y,\xbar,\ybar) d \pi(\xbar,\ybar),
\end{equation} 
where $\pi$ is a solution to the quadratic optimal transport problem, we may consider the classical transportation with the cost function $c_\pi$. If $\pi$ is the solution of this problem, we can obtain the $c_\pi$-cyclical monotonicity ${\pi}$. However, it is not trivial whether $\pi$ is a solution to the classical optimal transportation problem with cost function $c_\pi$. Still, we can show the cyclical monotonicity just using that we have a solution to the QOT.

\begin{Lem}\label{Lem: cyclical mono}
Let $\pi \in \Pi(\mu,\nu)$ be a solution to Problem \ref{prbm: qOT}. Define $c_\pi$ as \eqref{eqn: c_pi}. Then $\spt{\pi}$ is $c_\pi$-cyclically monotone.
\end{Lem}
\begin{proof}
By contradiction, suppose $\spt{\pi}$ is not $c_\pi$-cyclically monotone. Then $\exists (x_i, y_i) \in \spt{\pi}$ such that 
\begin{equation}\label{eqn: non cycl assume}
\sum_{i=1}^n c_\pi (x_i, y_i) > \sum_{i=1}^n c_\pi (x_i, y_{i+1}). 
\end{equation}
We define $\tilde{\pi}$ as in the proof of Lemma \ref{lem: cyclical monotone OT solution}. Then 
\begin{equation*}
\tilde{\pi} = \pi - \pi_{-} + \pi_{+},
\end{equation*}
where $\pi_{-} \leq \pi$ and $\pi_{-}$ and $\pi_{+}$ have the same marginal conditions so that $-\pi_{-} + \pi_{+}$ has zero marginals. Then $\tilde{\pi}_\epsilon = \pi + \epsilon( - \pi_{-} + \pi_{+}) \in \Pi(\mu,\nu)$ for any $\epsilon \ll 1$. Since $\pi$ is a solution (minimizer), we have
\begin{equation*}
\iint c d \tilde{\pi}_\epsilon d \tilde{\pi}_\epsilon - \iint c d \pi d\pi \geq 0.
\end{equation*}
Using symmetric condition of $c$, we obtain from above that
\begin{equation*}
2\epsilon \iint c d \pi d (\pi_+-\pi_-) + \epsilon^2 \iint c d\pi d (\pi_+-\pi_-) d (\pi_+-\pi_-) \geq 0.
\end{equation*}
Dividing by $\epsilon$ and taking $\epsilon \to 0$, we obtain 
\begin{equation*}
\iint c d \pi d (\pi_+-\pi_-) = \int c_\pi  d (\pi_+-\pi_-) \geq 0.
\end{equation*}
This, by construction of $\pi_-$ and $\pi_+$, implies 
\begin{equation*}
\sum_{i=1}^n c_\pi (x_i, y_i) \leq \sum_{i=1}^n c_\pi (x_i, y_{i+1})
\end{equation*}
which contradicts to \eqref{eqn: non cycl assume}. Therefore $\spt{\pi}$ has to be $c_\pi$-cyclically monotone.
\end{proof}

\begin{Rmk}\label{rmk: transpose cyclical mono}
We would like to point out that if $\pi$ is a solution to Problem \ref{prbm: qOT}, then $\pi$ is also a minimizer of $\iint c^* d\pi d\pi$ where $c^*(z,\zbar) = c(\zbar,z)$. If the cost function $c$ is not symmetric, that is, if $c(z,\zbar) \neq c(\zbar,z)$, then one can apply Lemma \ref{Lem: cyclical mono} with cost $c^*$ and obtain that $\spt{\pi}$ is also $c^*_\pi$-cyclically monotone. 
\end{Rmk}

In case of classical optimal transportation problem, $c$-cyclical monotonicity characterizes the solutions (Lemma \ref{lem: c-cyclical mono characterize solution}). We can try to use a similar argument to obtain the characterization. Suppose $\pi$ is a probability measure in $\Pi(\mu,\nu)$ which is $c_\pi$-cyclically monotone. Then, from Lemma \ref{lem: c-cyclical mono characterize solution}, we obtain that for any measure $\rho \in \Pi(\mu,\nu)$, we have
\begin{equation}\label{eqn: int c rho geq int c pi}
\int c_\pi d \rho \geq \int c_\pi d \pi.
\end{equation}
Then, we observe
\begin{align*}
\iint c d\rho d\rho  =& \iint c d(\pi + (\rho - \pi))(z) d(\pi + (\rho - \pi))(\zbar) \\
 =& \iint c d\pi(z) d\pi(\zbar) + \iint c d\pi(z) d (\rho - \pi)(\zbar) \\
&+ \iint c d (\rho - \pi)(z) d\pi(\zbar) + \iint c d (\rho - \pi)(z) d (\rho - \pi)(\zbar)\\
\geq& \iint c d\pi(z) d\pi(\zbar) + \iint c d(\rho - \pi)(z) d (\rho - \pi)(\zbar).
\end{align*}
where we have used Remark \ref{rmk: transpose cyclical mono} and \eqref{eqn: int c rho geq int c pi} in the last inequality. Hence, to claim that $\pi$ achieves the minumum, we need $\iint c d (\rho - \pi) d (\rho - \pi) \geq 0$.

\begin{Def}\label{def: positive semi-def cost}
Let $\Sigma(0,0)$ be the set of signed Borel measures on $X\times Y$ with zero marginals.
\begin{equation*}
\Sigma(0,0) = \{ \rho \in \Meas{X\times Y} | \Proj{X}_\sharp \rho = 0, \Proj{Y}_\sharp \rho = 0 \}.
\end{equation*}
A cost function $c: Z^2 \to \R$ is called \emph{positive semi-definite} if the following inequality holds for any measure $\rho \in \Sigma(0,0)$.
\begin{equation*}
\iint c(z, \zbar) d \rho(z) d \rho(\zbar) \geq 0.
\end{equation*}
\end{Def}

\begin{Rmk}\label{Rmk: positive nonsymmetric}
If $c$ is positive semi-definite, then one can show that for any $\rho_0, \rho_1 \in \Sigma(0,0)$, 
\begin{equation*}
\iint c d \rho_0(z) d\rho_1 (\zbar) = \iint c d \rho_1(z) d \rho_0(\zbar).
\end{equation*}
However, this does not mean that we have $c(z,\zbar) = c(\zbar,z)$. In particular, positive semi-definite \emph{does not} imply 
\begin{equation*}
\iint c d \pi_0 (z) d\pi_1(\zbar) = \iint c d\pi_1(z) d\pi_0(\zbar).
\end{equation*}
A simple example of a cost function that is positive semi-definite but not symmetric can be constructed on the two point space $X=Y=\{ 0,1\}$. Let $c(x,y,\xbar,\ybar)$ be positive when the sum $x+y+\xbar+\ybar$ is even and 0 otherwise. Any such cost $c$ is a positive semi-definite, but it is not necessarily symmetric.
\end{Rmk}

If a cost function $c$ is positive semi-definite, then we can show that $c_\pi$ characterizes the solutions using the argument above.

\begin{Prop}\label{prop: characterization}
Let $c$ be a positive semi-definite cost function, and let $\pi \in \Pi(\mu,\nu)$ be a probability measure that is $c_\pi$-cyclically monotone. Then $\pi$ is a solution to Problem \ref{prbm: qOT}.
\end{Prop}
\begin{proof}
From the argument above Definition \ref{def: positive semi-def cost}, we have 
\begin{equation*}
\iint c d\rho d\rho \geq \iint c d\pi d\pi + \iint c d(\rho-\pi) d (\rho -\pi).
\end{equation*}
Since $c$ is positive semi-definite and $\rho-\pi \in \Sigma(0,0)$, the last term on the right hand side is non-negative. Therefore, we obtain
\begin{equation*}
\iint c d\rho d\rho \geq \iint c d\pi d\pi,
\end{equation*}
for any $\rho \in \Pi(\mu,\nu)$. This implies that $\pi$ achieves the infimum of Problem \ref{prbm: qOT}.
\end{proof}

Using the characterization from Proposition \ref{prop: characterization}, we can prove a qualitative stability result. 

\begin{Prop}\label{prop: qot stability}
Let $c$ be a continuous, positive semi-definite cost function. Let $ \mu_n $ and $ \nu_n$ be the sequences of probability measures on $X$ and $Y$ respectively, and let $\pi_n$ be a solution to Problem \ref{prbm: qOT} with marginals $\mu_n$ and $\nu_n$.  Suppose $\mu_n$ and $\nu_n$ converges weakly to $\mu_0$ and $\nu_0$ respectively. Then, up to a subsequence, $\pi_n$ converges weakly to $\pi_0$ which is a solution to Problem \ref{prbm: qOT} with marginals $\mu_0$ and $\nu_0$. 
\end{Prop}
\begin{proof}
Note first that the weak convergence of the sequences $\mu_n$ and $\nu_n$ implies that $\{ \pi_n \}$ is tight. Therefore, up to a subsequence, we can assume that $\pi_n$ converges weakly to $\pi_0$. Moreover, the marginal conditions of $\pi_n$ imply that $\pi_0 \in \Pi(\mu_0, \nu_0)$. Denote $c_n(z) = c_{\pi_n}(z) = \int c(z,\zbar) d\pi_n (\zbar)$. Proposition \ref{prop: characterization} shows that $\pi_n$  is a Kantorovich solution to the classical optimal transportation problem with marginals $\mu_n$ and $\nu_n$. Also, Weak convergence of $\pi_n$ to $\pi_0$ implies that $c_n$ converges to $c_0$ uniformly. Then we use \ref{prop: ot stability} to obtain that $\pi_0$ is a Kantorovich solution to the classical optimal transportation problem with cost function $c_0 = c_{\pi_0}$ and marginals $\mu_0$ and $\nu_0$. Then Lemma \ref{lem: cyclical monotone OT solution} shows that $\pi_0$ is $c_{\pi_0}$-monotone and we deduce from Proposition \ref{prop: characterization} that $\pi_0$ is a solution to Problem \ref{prbm: qOT} with marginals $\mu_0$ and $\nu_0$.
\end{proof}

Moreover, with positive semi-definite cost function, the total cost $\iint c d\pi d\pi$ becomes a convex function of $\pi$. 

\begin{Lem}\label{lem: conv total cost}
Let $c$ be a positive semi-definite cost function. Then the total cost is a convex function on $\Pi(\mu,\nu)$.
\end{Lem}
\begin{proof}
Let $\pi_0, \pi_1 \in \Pi(\mu,\nu)$ and let $\pi_t = t \pi_0 + (1-t)\pi_1 \in \Pi(\mu,\nu)$. Then 
\begin{align*}
  \iint c d\pi_t d\pi_t 
 = & t^2 \iint c d\pi_0 d\pi_0 + (1-t)^2 \iint c d\pi_1 d\pi_1 \\
 & + t(1-t) \iint c d\pi_0  d\pi_1 + t(1-t) \iint c d\pi_1  d\pi_2 
\end{align*}
and therefore
\begin{align*}
&\frac{d^2}{dt^2} \iint c d\pi_t d\pi_t\\
  =& 2\left( \iint c d\pi_0 d\pi_0 + \iint c d\pi_1 d\pi_1 -\iint c d\pi_0 d\pi_1 - \iint c d\pi_1 d\pi_0 \right) \\
 =& 2\iint c d(\pi_1 - \pi_0) d(\pi_1-\pi_0) \geq 0.
\end{align*}
This implies the convexity of the total cost.
\end{proof}

\begin{Rmk}\label{rmk: positive definite cost unique sol}
If we have a positive definite cost function, that is, if
\begin{equation*}
\iint c d\rho d\rho >0, \forall \rho \in \Sigma(0,0),
\end{equation*}
then by the proof of Lemma \ref{lem: conv total cost} the total cost is strictly convex and the solution to QOT will be unique. In such case, the whole sequence $\pi_n$ in Proposition \ref{prop: qot stability} will converge to the unique solution $\pi_0$.
\end{Rmk}

One can view $(\rho, \overline{\rho}) \mapsto \iint c d\rho d \overline{\rho}$ as a bi-linear form defined on the space of Borel measures. Then the positive semi-definite cost functions corresponds to the positive semi-definite quadratic forms on $\Sigma(0,0)$. Positive semi-definite operators on Euclidean spaces or complex Hilbert spaces have their square root (For example, see Theorem 7.16 in \cite{conway2000course}, which shows existence of square root operator for a compact positive operator on a complex Hilbert space). In our case, however, the space of Borel measures does not have a Hilbert structure. Therefore, we make another definition for the positive semi-definite cost functions that have its square root in the following sense.

\begin{Def}\label{def: squared cost}
a cost function $c: Z^2 \to \R$ is called \emph{squared} if there is a metric measure space $(W,d_W, \omega)$, and a function $S:Z \times W \to \R$ such that
\begin{equation*}
c(z,\bar{z}) = \int S(z,w)S(\zbar,w) d\omega(w).
\end{equation*}
We call $S$ a square root of $c$. 
\end{Def}

\begin{Rmk}
We have used the probability measure $\omega$ in Definition \ref{def: squared cost} for computational benefit. In fact, as long as the integrals against $\omega$ are finite, one can use a positive measure that is not necessarily a probability. However, there could be a constant factor difference in the formulas.
\end{Rmk}

When we assume some regularity on the square root $S$ of $c$, we denote regularity of $S$ in front of the term squared. For instance, if $S$ is $C^{0,\alpha}(Z\times W)$, then we say that $c$ is a $C^{0,\alpha}$-squared cost.

\begin{Lem}
Continuously squared costs are positive semi-definite.
\end{Lem}
\begin{proof}
Suppose $c(z, \zbar) = \int S(z,w)S(\zbar,w)d\omega$ be a squared cost function with $S$ continuous. Then, for any Borel measure $\rho$ on $Z$, we have
\begin{align*}
\iint c d\rho d\rho & = \iint \left( \int S(z,w)S(\zbar,w) d\omega(w) \right) d\rho(z) d\rho(\zbar) \\
& = \int \left( \int S(z,w) d\rho(z) \right)^2 d\omega(w) \geq 0.
\end{align*}
\end{proof}

An advantage of squared cost is that we can view the cost $\iint c d\pi d\pi$ as an $L^2$-norm of some function defined using the measure $\pi$. We denote
\begin{equation*}
S_\pi(w) = \int S(z,w) d\pi(z).
\end{equation*}
Then we obtain
\begin{equation*}
\iint c d\pi d\pi = \| S_\pi \|^2_{L^2}.
\end{equation*}
Hence QOT asks to find a measure $\pi \in \Pi(\mu,\nu)$ such that $S_\pi$ is a closest to $0$. We denote the set of $S_\pi$ by $S_{\Pi(\mu,\nu)}$.
\begin{equation*}
S_{\Pi(\mu,\nu)} = \left\{ S_\pi(w) = \int S(z,w) d\pi(w) | \pi \in \Pi(\mu,\nu) \right\}.
\end{equation*}
The total cost is then 
\begin{equation*}
\inf_{\Pi(\mu,\nu)} \iint c d\pi d\pi = \mathrm{dist}(S_{\Pi(\mu,\nu)}, 0)^2.
\end{equation*}
The set $S_{\Pi(\mu,\nu)}$ enjoys some important properties. 

\begin{Lem} 
Suppose $S$ is continuous, then \hfill

\begin{tabular}{cl}
$(1)$& $S_{\Pi(\mu,\nu)}$ is convex, \\
$(2)$& $S_{\Pi(\mu,\nu)}$ is strongly compact in $L^2(\omega)$. 
\end{tabular}
\end{Lem}
\begin{proof}
(1): We compute 
\begin{align*}
tS_{\pi_1} + (1-t) S_{\pi_0} & = t \int S(z,w) d \pi_1(z) + (1-t) \int S(z,w) d\pi_0(z) \\
& = \int S(z,w) d (t \pi_0 + (1-t) \pi_1)(z) \\
& = S_{t\pi_1 + (1-t) \pi_0}
\end{align*}
and $t\pi_1 + (1-t) \pi_0 \in \Pi(\mu,\nu)$.\\
(2): Let $\{ g_k \}_{k=1}^\infty$ be a sequence in $S_{\Pi(\mu,\nu)}$. Then $g_k = S_{\pi_k}$ for some $\pi_k \in S_{\Pi(\mu,\nu)}$. As $\Pi(\mu,\nu)$ is weakly compact, up to a subsequence (which we still denote with $\pi_k$), there exists a $\pi \in \Pi(\mu,\nu)$ such that $\pi_k$ converges to $\pi$ weakly. Define $g = S_\pi$. We claim that $g_k$ converges to $g$ strongly. Note first that 
\begin{equation*}
| g_k(w) - g(w) | \leq \int |S(z,w)| d \pi_k(w) + \int |S(z,w)| d\pi(w) \leq 2 \sup S.
\end{equation*}
Then we compute
\begin{align*}
\lim_{k \to \infty} \| g_k - g \|_{L^2(\omega)}^2 & = \lim_{k \to \infty} \int | g_k - g |^2 d\omega \\
& = \int \lim_{k \to \infty} | g_k - g |^2 d\omega \\
& = \int \lim_{k \to \infty} \left( \int S(z,w) d \pi_k - \int S(z,w) d\pi \right)^2 d\omega \\
& =  \int 0 d\omega = 0
\end{align*}
where we have used the dominated convergence theorem in the second equality and weak convergence of $\pi_k$ to $\pi$ in the fourth equality. Since $g \in \Pi(\mu,\nu)$ by definition, we obtain that $S_{\Pi(\mu,\nu)}$ is compact. 
\end{proof}

We use this structure of the squared cost functions to formulate the dual problem in the next section.

\section{Duality for the squared costs}\label{sec: dual}
Since we view QOT as a version of the classical optimal transportation problem, one natural question to ask is to find a duality formula. Kantorovich duality for the classical optimal transportation problem states
\begin{equation}\label{eqn: Kantorovich duality}
\inf_{\pi \in \Pi(\mu,\nu)} \int c(x,y) d\pi = \sup_{(\pi, \psi) \in \Phi_c}\left( \int \phi d \mu + \int \psi d\mu \right),
\end{equation}
where
\begin{equation*}
\Phi_c = \{ (\phi, \psi) L^1(\mu) \times L^1 (\nu) | \phi(x) + \psi(y) \leq c(x,y) \textrm{ for } \mu \otimes \nu \textrm{ a.e. } (x,y) \}.
\end{equation*}
Kantorovich duality can be seen as a strong duality of Lagrangian dual problem of the minimization problem Problem \ref{prbm: OT}. A work about finding a dual problem in this view point (Strong duality of Lagrangian dual) of the quadratic optimal transportation problem can be found in \cite{cabrera2022optimal}.

In this section, we present a different duality formula for Problem \ref{prbm: qOT} with a squared cost function $c$. 

Using the dual formulation of $L^2$ norm, we see that
\begin{align*}
\inf_{\Pi(\mu,\nu)} \sqrt{\iint c d\pi d\pi} & = \inf_{\Pi(\mu,\nu)} \| S_\pi \|_{L^2} \\
& = \inf_{\Pi(\mu,\nu)} \sup_{\| f \|_{L^2} \leq 1} \langle S_\pi, f \rangle_{L^2}. \numberthis \label{eqn: using L2 dual}
\end{align*}
Note that we have used positiveness of the total cost $\iint c d\pi d\pi$ to take the square root. Suppose for now that we can exchange the $\sup$ and $\inf$ in the last line. Then we obtain
\begin{align*}
& \sup_{\| f \|_{L^2}\leq 1} \inf_{\Pi(\mu,\nu)} \langle S_\pi, f \rangle_{L^2}\\
=&\sup_{\| f \|_{L^2}\leq 1} \inf_{\Pi(\mu,\nu)} \int \left( \int S(z,w) d\pi(z) \right) f(w) d\omega(w) \\
=& \sup_{\| f \|_{L^2}\leq 1} \inf_{\Pi(\mu,\nu)} \int \left( \int S(x,y,w) f(w) d\omega(w) \right) d\pi(x,y).
\end{align*}
Denote
\begin{equation*}
S^f (x,y) = \int S(x,y,w) f(w) d\omega(w).
\end{equation*}
Then we obtain
\begin{equation}\label{eqn: sqrt is ot}
\inf_{\Pi(\mu,\nu)} \sqrt{\iint c d\pi d\pi} = \sup_{\| f \|_{L^2}\leq 1} \inf_{\Pi(\mu,\nu)} \int S^f(x,y) d\pi(x,y).
\end{equation}
We can see that the $\inf$ part inside $\sup$ on the right hand side is in fact the classical optimal transportation problem Problem \ref{prbm: OT} with the cost function $S^f$. Hence we can invoke Kantorovich duality \eqref{eqn: Kantorovich duality} and obtain
\begin{equation*}
\inf_{\Pi(\mu,\nu)} \sqrt{\iint c d\pi d\pi}=\sup_{\| f \|_{L^2}\leq 1} \sup_{\Phi_{S^f}}\left(\int \phi d\mu + \int \psi d\nu \right),
\end{equation*}
which suggests a dual problem with a very similar formula with \eqref{eqn: Kantorovich duality}, but includes another function $f$ that needs to be considered. 

We show that we can indeed swab the $\inf$ and $\sup$.

\begin{Lem}\label{lem: inf sup swab}
\begin{equation}\label{eqn: inf sup swab}
\inf_{\Pi(\mu,\nu)} \sup_{\| f \|_{L^2} \leq 1} \langle S_\pi, f \rangle_{L^2} =  \sup_{\| f \|_{L^2} \leq 1} \inf_{\Pi(\mu,\nu)} \langle S_\pi, f \rangle_{L^2}.
\end{equation}
\end{Lem}
\begin{proof}

We first consider the case $0 \in S_{\Pi(\mu,\nu)}$. If $0 \in S_{\Pi(\mu,\nu)}$ then
\begin{equation*}
\inf_{\Pi(\mu,\nu)} \sup_{\| f \|_{L^2} \leq 1} \langle S_\pi, f \rangle_{L^2} = \inf_{\Pi(\mu,\nu)} \| S_\pi \| = 0.
\end{equation*}
On the other hand, we have
\begin{equation*}
\inf_{\Pi(\mu,\nu)} \langle S_\pi, f \rangle_{L^2} \leq \langle 0 , f \rangle_{L^2}=0
\end{equation*}
so that $\sup_{\| f \|_{L^2} \leq 1} \inf_{\Pi(\mu,\nu)} \langle S_\pi, f \rangle_{L^2} \leq 0$. Since the equality can be attained by choosing $f=0$, we see the right hand side of \eqref{eqn: inf sup swab} is also 0. Therefore we obtain the equality when $0 \in S_{\Pi(\mu,\nu)}$.

Next, we consider the case when $0 \not\in S_{\Pi(\mu,\nu)}$. We note that the left hand side of \eqref{eqn: inf sup swab} is $\mathrm{dist}(S_{\Pi(\mu,\nu)},0)$. Since we always have that $\inf \sup$ is greater than or equal to $\sup \inf$, it is enough show that the right hand side of \eqref{eqn: inf sup swab} is bigger than or equal to $\mathrm{dist}(S_{\Pi(\mu,\nu)},0)$. Let $d = \mathrm{dist}(S_{\Pi(\mu,\nu)},0)$. Then we observe that $B^{L^2(\omega)}_d(0) \cap S_{\Pi(\mu,\nu)} = \emptyset$. By Hahn-Banach Theorem, there exists a linear functional $\lambda$ such that
\begin{equation*}
\inf_{g \in S_{\Pi(\mu,\nu)}} \lambda(g) \geq \sup_{f \in B^{L^2(\omega)}_d(0)}\lambda(f).
\end{equation*}
Using Riesz representation Theorem and abbreviating the notation, we write $\lambda(g) = \langle \lambda , g \rangle$ with $\lambda \in L^2(\omega)$. Multiplying a constant if needed, we can assume that $\| \lambda \|_{L^2} = 1$. Then
\begin{align*}
\sup_{f \in B^{L^2(\omega)}_1(0)} \inf_{g \in S_{\Pi(\mu,\nu)}} \langle f,g \rangle & \geq \inf_{g \in S_{\Pi(\mu,\nu)}} \langle \lambda , g \rangle \\
& \geq \sup_{f \in B_d(0)} \langle \lambda, f \rangle \\
& \geq \sup_{r < d} \langle \lambda, r\lambda \rangle = d.
\end{align*}
Therefore, the right hand side of \eqref{eqn: inf sup swab} is bigger than or equal to $d = \mathrm{dist}(S_{\Pi(\mu,\nu)}, 0)$ which is equal to the left hand side of \eqref{eqn: inf sup swab}. Hence we obtain \eqref{eqn: inf sup swab}.
\end{proof}

Lemma \ref{lem: inf sup swab} with the discussion above the lemma shows the following duality formula.

\begin{Thm}\label{thm: duality}
Suppose $c$ is continuously squared cost function with square root $S$. Then
\begin{equation}\label{eqn: duality}
\inf_{\Pi(\mu,\nu)} \sqrt{\iint c d\pi d\pi} = \sup_{\substack{\| f \|_{L^2}\leq 1 \\ (\phi, \psi) \in \Phi_{S^f}}} \left(\int \phi d\mu + \int \psi d\nu \right).
\end{equation}
\end{Thm}

The equation \eqref{eqn: duality} shows that if the cost function is squared in the sense of Definition \ref{def: squared cost}, then indeed the QOT can be considered as a square of a certain classical optimal transportation problem. Since there is an extra function $f \in L^2(\omega)$ to consider, however, the proof for the existence of a solution to the dual problem will not be carried directly from the classical optimal transportation case.

\begin{Prop}\label{prop: existence of dual sol}
Suppose $c$ is a uniformly continuously squared cost function. Then solutions to the dual problem \eqref{eqn: duality} exist.
\end{Prop}
\begin{proof}
Since $c$ is uniformly continuously squared cost, the square root $S$ is uniformly continuous, i.e. for any $\epsilon>0$, there exists $\delta>0$ such that if $ d_{Z\times W}((z_0,w_0),(z_1,w_1)) < \delta$, then $|S(z_1,w_1)-S(z_0,w_0)|<\epsilon$. Then we observe that for any $z_0, z_1 \in Z$ such that $d_Z(z_0,z_1)< \delta$,
\begin{align*}
|S^f(z_1) - S^f(z_0)| & = \left| \int (S(z_1,w)-S(z_0,w))f(w)d\omega(w) \right| \\
& \leq \left(\int (S(z_1,w)-S(z_0,w)) d\omega(w) \right)^{\frac{1}{2}} \left( \int f^2(w) d\omega(w) \right) ^{\frac{1}{2}} \\
& \leq \epsilon.
\end{align*} 
For any $f \in \overline{B^{L^2(\omega)}_1(0)}$. Hence the family of functions $\{ S^f | f \in \overline{B^{L^2(\omega)}_1(0)} \}$ is equi-uniformly continuous. Now Fix any $f \in \overline{B^{L^2(\omega)}_1(0)}$. We show that the supremum
\begin{equation}\label{eqn: dual sup f fixed}
\sup_{(\phi,\psi)\in\Phi_{S^f}} \left( \int \phi d\mu + \int \psi d\nu \right)
\end{equation}
is achieved by a pair of functions $(\phi_f,\psi_f) \in \Phi_{S^f}$. Let $(\phi_k, \psi_k) \in \Phi_{S^f}$, $k \in \N$ be a maximizing sequence. Then, we define
\begin{align*}
\psi^*_k(y) & := \inf_{x \in X} \{ S^f(x,y) - \phi_k(x)\}, \\
\phi^*_k(x) & := \inf_{y \in Y} \{ S^f(x,y) - \psi^*_k(y)\}.
\end{align*}
Note that $(\phi_k,\psi_k) \in \Phi_{S^f}$ implies $\psi_k(y) \leq \psi_k^*(y)$, $\phi(x) \leq \phi_k^*(x)$ and $(\phi_k^*, \psi_k^*) \in \Phi_{S^f}$. 
We show that the sequence $\phi^*_k$ and $\psi^*_k$ are equi-continuous. Let $\epsilon>0$ be arbitrary and let $\delta>0$ be a positive number such that if $|x_1 -x_0| < \delta$ then $|S^f(x_1,y)-S^f(x_0,y)| < \epsilon/2$. Also, there exists $y_{k,\epsilon}\in Y$ such that
\begin{equation*}
| \psi_k^*(x_0) - (S^f(x_0,y_{k,\epsilon})-\psi_k(y_{k,\epsilon}))| < \frac{\epsilon}{2}.
\end{equation*}
Then we observe
\begin{align*}
\psi_k^*(x_1) -\psi_k^*(x_0) & = \inf \{ S^f (x_1,y) - \psi_k(y) \} - \inf \{ S^f (x_0,y) - \psi_k(y) \} \\
& < \inf \{ S^f (x_1,y) - \psi_k(y) \} -( S^f(x_0,y_{k,\epsilon}) - \psi_k(y_{k,\epsilon}) ) + \frac{\epsilon}{2} \\
& \leq (S^f(x_1, y_{k,\epsilon}) - \psi_k(y_{k,\epsilon})) -( S^f(x_0,y_{k,\epsilon}) - \psi_k(y_{k,\epsilon}) ) + \frac{\epsilon}{2} \\
& = S^f(x_1, y_{k,\epsilon} ) - S^f(x_0, y_{k,\epsilon}) + \frac{\epsilon}{2} \\
& < \epsilon.
\end{align*}
$\psi_k^*(x_0) - \psi_k^*(x_1) < \epsilon$ can also be proved similarly using the same $\delta$. This implies that for any $\epsilon>0$ and $k\in \N$, there exists $\delta$ such that if $| x_1 - x_0| < \delta$ then
\begin{equation*}
|\psi_k^*(x_1) - \psi_k^*(x_0)| < \epsilon.
\end{equation*}
Hence the sequence $\{ \psi_k^* \}$ is equi-continuous. Equi-continuity of the sequence $\{ \phi_k^* \}$ can also be proved similarly using the same $\delta$, and we obtain that the sequences $\{\phi_k^*\}$ and $\{\psi_k^*\}$ are equi-continuous. 

Now we use Arzela-Ascoli to obtain that, up to subsequences, $\{ \phi_k^* \}$ and $\{ \psi_k^* \}$ uniformly converges to $\phi_f$ and $\psi_f$ respectively. Since $(\phi_k^*, \psi_k^*)$ was a maximizing sequence, we obtain
\begin{equation*}
\sup \left( \int \phi d\mu + \int \psi d\nu \right) = \lim_{k \to \infty} \left( \int \phi_k^* d\mu + \int \psi_k^* d\nu \right) = \int \phi_f d\mu + \int \psi_f d\nu.
\end{equation*}
Therefore, the pair of the functions $(\phi_f, \psi_f)$ achieves the supremum \eqref{eqn: dual sup f fixed}.

Now, let $f_k \in \overline{B^{L^2(\omega)}_1(0)}$ be a maximizing sequence of the quantity $\int \phi_f d\mu + \int \psi_f d\nu$. By Banach-Alouglu, $\overline{B^{L^2(\omega)}_1(0)}$, the closed unit ball in the Hilbert space $L^2(\omega)$, is weakly compact. Hence, up to a subsequece, $f_k$ converges to some $f_0$ weakly. Also, from the construction of the pairs $(\phi_f, \psi_f)$, the family of pairs of functions $\{ (\phi_f, \psi_f) | f \in\overline{B^{L^2(\omega)}_1(0)} \}$ is equi-uniformly continuous. Hence, by Arzela-Ascoli again, up to a subsequence $(\phi_{f_k}, \psi_{f_k})$ uniformly converges to $(\phi_0, \psi_0)$. One can easily see that $(\phi_0,\psi_0) \in \Phi_{f_0}$. Then we observe that
\begin{align*}
\sup_{f \in \overline{B^{L^2(\omega)}_1(0)}} \left( \int \phi_f d\mu + \int \psi_f d\nu \right) & = \lim_{k \to \infty} \left( \int \phi_{f_k} d\mu + \int \psi_{f_k} d\nu \right) \\
& = \int \phi_0 d\mu+ \int \psi_0 d\nu \\
& \leq \int \phi_{f_0} d\mu + \int \psi_{f_0} d\nu.
\end{align*}
The first term $\sup_{f \in \overline{B^{L^2(\omega)}_1(0)}} \left( \int \phi_f d\mu + \int \psi_f d\nu \right)$ is equal to the right hand side of \eqref{eqn: duality}, $\sup_{\substack{\| f \|_{L^2}\leq 1 \\ (\phi, \psi) \in \Phi_{S^f}}} \left(\int \phi d\mu + \int \psi d\nu \right)$, i.e., we have
\begin{equation*}
\sup_{\substack{\| f \|_{L^2}\leq 1 \\ (\phi, \psi) \in \Phi_{S^f}}} \left(\int \phi d\mu + \int \psi d\nu \right) \leq \int \phi_{f_0} d\mu + \int \psi_{f_0} d\nu.
\end{equation*}
This with $(\phi_{f_0}, \psi_{f_0})\in \Phi_{f_0}$ implies that the triple $(f_0, \phi_{f_0}, \psi_{f_0})$ is a solution to the dual problem \eqref{eqn: duality}.
\end{proof}

\begin{Rmk}\label{rmk: opt f}
In the proof of Proposition \ref{prop: existence of dual sol}, we regarded the dual problem as an independent problem from QOT. However, existence of optimal $f$ in the dual problem \eqref{eqn: duality} can also be deduced from the existence of a Kantorovich solution Proposition \ref{prop: exist opt plan}. From \eqref{eqn: using L2 dual}, we observe that the optimal $f$ has to be a rescaling of $S_\pi$ for a Kantorovich solution $\pi$. 
\end{Rmk}

We do not know if the solution of QOT is unique unless we have positive definite as described in Remark \ref{rmk: positive definite cost unique sol}. The optimal function $f$ from the dual problem \eqref{eqn: duality}, however, is unique due to the strict convexity of the unit ball $\overline{B^{L^2(\omega)}_1(0)}$ in $L^2(\omega)$. 

\begin{Prop}
The optimal $f$ is unique.
\end{Prop}
\begin{proof}
We first show that if $f$ is optimal in the dual problem, then we must have $\| f \|_{L^2} = 1$. Indeed, if $(\phi,\psi) \in \Phi_{S^f}$, then $(k\phi, k\psi) \in \Psi_{S^{kf}}$ for any $k>0$ and therefore we can observe
\begin{equation*}
\sup_{(\phi,\psi)\in \Phi_{S^{kf}}}\left( \int\phi d\mu + \int \psi d\nu \right) = k \times \sup_{(\phi,\psi)\in \Phi_{S^f}}\left( \int \phi d\mu + \int \phi d\nu \right).
\end{equation*}
Hence if $\| f \|_{L^2} <1$, then we can use $\tilde{f} = f/\|f\|_{L^2}$ to obtain a greater quantity in the dual problem. Then the uniqueness of optimal $f$ follows from the strict convexity of the unit ball $\overline{B^{L^2(\omega)}_1(0)}$ in $L^2(\omega)$. Suppose there exist two different optimal $f$, say $f_0$ and $f_1$. Then we have $\| f_0 \|_{L^2}  =\| f_1 \|_{L^2} = 1$. Define $f_t = t f_1 + (1-t) f_0$. Then we have $\| f_t \| < 1$ for $t \in (0,1)$. On the other hand, for $(\phi_i, \psi_i) \in \Phi_{S^{f_i}}$, $i=0,1$, we can obtain
\begin{align*}
\phi_t(x) + \psi_t(y) & := ( t\phi_1 (x) + (1-t) \phi_0(x) ) + (t \psi_1(y) + (1-t)\psi_0(y)) \\
& = t ( \phi_1 (x) + \psi_1(y) ) + (1-t)(\phi_0(x) + \psi_0(y) ) \\
& \leq t \int S(x,y,w) f_1(w) d\omega(w) + (1-t) \int S(x,y,w) f_0(w) d\omega(w) \\
& = \int S(x,y,w) f_t (w) d\omega(w) = S^{f_t}(x,y),
\end{align*}
so that $(\phi_t , \psi_t) \in \Phi_{S^{f_t}}$. Therefore
\begin{align*}
&\sup_{(\phi_0 , \psi_0)\in \Phi_{S^{f_0}}} \left( \int \phi_0 d\mu + \int \psi_0 d\nu \right)+ \sup_{(\phi_1 , \psi_1)\in \Phi_{S^{f_1}}} \left( \int \phi_1 d\mu + \int \psi_1 d\nu \right)\\
\leq & \sup_{(\phi,\psi)\in \Phi_{S^{f_t}}} \left( \int \phi d\mu + \int \psi d\nu \right) \\
= & \| f_t \| \sup_{(\phi,\psi)\in \Phi_{S^{\tilde{f}_t}}} \left( \int \phi d\mu + \int \psi d\nu \right) \\
< & \sup_{(\phi,\psi)\in \Phi_{S^{\tilde{f}_t}}} \left( \int \phi d\mu + \int \psi d\nu \right).
\end{align*}
where $\tilde{f}_t = {f_t}/{\| f_t \|_{L^2}} $. Note that we have used $\| f_t \|_{L^2} < 1$ in the last inequality. This contradicts to the assumption that both $f_0$ and $f_1$ are optimal and therefore optimal $f$ must be unique.
\end{proof}

\bibliographystyle{plain}

\bibliography{QuadraticOT.bib}

\newpage

\begin{appendices}

\section{Classical optimal transportation problem}\label{sec: OT}
In this section of appendix, we review some optimal transportation theories. We will omit most of the proofs in this section unless we need the proof in the paper. We refer the proofs to \cite{villani2008optimal} when its omitted. We will use the same notations as in the body of the paper in most case, but we will use the continuous cost function $c:X\times Y \to \R$. 

Classical optimal transportation problem seeks for a map $T: X \to Y$ which minimize the total cost.

\begin{Prbm*}[OT-Monge]\customlabel{prbm: Monge OT}{(OT-Monge)}
Find a map $T:X \to Y$ which satisfies $T_\sharp \mu = \nu$ and realizes the following infimum:
\begin{equation*}
\inf_{T_\sharp \mu = \nu} \int c (x, T(x)) d\mu(x).
\end{equation*}
\end{Prbm*}

A relaxed version of Problem \ref{prbm: Monge OT} can be formed using measures instead of the map $T$.

\begin{Prbm*}[OT-Kantorovich]\customlabel{prbm: Kant OT}{(OT-Kantorovich)}
Find a measure $\pi \in \Pi(\mu,\nu)$ which realizes the following infimum:
\begin{equation*}
\inf_{\pi \in \Pi(\mu,\nu)} \int c (x,y) d\pi(x,y).
\end{equation*}
\end{Prbm*}

Problem \ref{prbm: Kant OT} always admits a solution. This can be proved with a simple application of the Prokhorov's theorem.

\begin{Lem}[Prokhorov's theorem]\label{lem: prokhorov}
A set $A$ of probability measures on $X$ is called tight if for any $\epsilon>0$, there exists a compact subset $K$ of $X$ such that for any $\rho \in A$, $\rho[X \setminus K] < \epsilon$. Then a set of probability measures is tight if and only if the set is weakly sequentially pre-compact. 
\end{Lem}

\begin{Prop}
Problem \ref{prbm: Kant OT} admits a solution.
\end{Prop}

If $T$ is a Monge solution to the classical optimal transportation problem, then restriction of $T$ is also a Monge solution with corresponding marginal conditions.

\begin{Prop}\label{prop: ot local}
Suppose $\mu'$ is a probability measure on $X$ such that $\mu' \ll \mu$ and let $\nu' = T_\sharp \mu'$. Then $T$ is a solution to Problem \ref{prbm: Monge OT} with source and target measures $\mu'$ and $\nu'$.
\end{Prop}

The solutions of Problem \ref{prbm: Kant OT} satisfy a special monotone structure called $c$-cyclical monotonicity.

\begin{Def}\label{def: cyclical monotone}
A set $A \subset Z$ is called $c$-cyclically monotone if, for any $(x_i, y_i) \in A$, $i=1, \cdots, n$, we have
\begin{equation*}
\sum_{i=1}^n c(x_i , y_i ) \leq \sum_{i=1}^n c (x_i, y_{i+1})
\end{equation*}
with convention $y_{n+1} = y_1$. A probability measure $\pi \in \Prob{Z}$ is called $c$-cyclically monotone if $\spt{\pi}$ is $c$-cyclically monotone.
\end{Def}

\begin{Lem}\label{lem: cyclical monotone OT solution}
Solutions of Problem \ref{prbm: Kant OT} are $c$-cyclically monotone.
\end{Lem}
\begin{proof}
Suppose there is a solution $\pi$ of Problem \ref{prbm: Kant OT} that is not $c$-cyclically monotone. Then there exist $n \in \N$ and $(x_i, y_i) \in \spt{\pi}$, $1 \leq i \leq n$ such that
\begin{equation*}
\sum_{i=1}^n c(x_i , y_i ) > \sum_{i=1}^n c (x_i, y_{i+1}).
\end{equation*}
Since $c$ is continuous, there exist $r>0$ such that 
\begin{equation}\label{eqn: non cyclic ineq}
\sum_{i=1}^n c(\xbar_i , \ybar_i ) > \sum_{i=1}^n c (\xbar_i, \ybar_{i+1})
\end{equation}
for any $(\xbar_i, \ybar_i) \in B_r^X(x_i)\times B_r^Y(y_i)$. Also, as $(x_i, y_i) \spt{\pi}$, we have $\pi[B_r^X(x_i)\times B_r^Y(y_i)] >0$. Let $\epsilon$ be a small positive number such that $\epsilon < \min_{1 \leq i \leq n} \pi[B_r^X(x_i)\times B_r^Y(y_i)]$ We define probability measures $\pi_i$ by
\begin{equation*}
\pi_i = \frac{1}{\pi[B_r^X(x_i)\times B_r^Y(y_i)]} \pi \lfloor_{B_r^X(x_i)\times B_r^Y(y_i)},
\end{equation*}
and a measure $\pi_-$ by
\begin{equation*}
\pi_- = \epsilon \sum_{i=1}^n \pi_i.
\end{equation*}
Let $\gamma = \pi_1 \otimes \cdots \otimes \pi_n$. Then $\gamma$ is a probability measure in $Z^n$ and $\pi_i  = \Proj{i}_\sharp \gamma$ where $\Proj{i} : Z^n \to Z$ is the projection to $i$-th coordinate. Since $Z = X \times Y$, we have $\Proj{i} = (U_i, V_i)$ for some $U_i: Z^n \to Z$ and $V_i : Z^n \to Z$. We define $\tilde{\pi}_i$ by 
\begin{equation*}
 \tilde{\pi}_i = (U_i, V_{i+1})_\sharp \gamma
\end{equation*}
and a measure $\pi_+$ by
\begin{equation*}
\pi_+ = \epsilon \sum_{i=1}^n \tilde{\pi}_i.
\end{equation*}
Finally, define 
\begin{equation*}
\tilde{\pi} = \pi - \pi_- + \pi_+.
\end{equation*}
From the choice of $\epsilon$ and construction of $\pi_-$ and $\pi_+$, $\tilde{\pi}$ is a non-negative measure. Also, $\pi_-$ and $\pi_+$ have both $X$-marginal and $Y$-marginal the same respectively. In particular, $\tilde{\pi}$ is a probability measure with the same marginal condition with $\pi$, and hence $\tilde{\pi} \in \Pi(\mu,\nu)$. Then we compute
\begin{align*}
& \int c d\tilde{\pi} \\
& = \int c d\pi - \epsilon \int c d \pi_- + \epsilon \int c d \pi_+ \\
&  = \int c d\pi - \epsilon \left( \sum_{i=1}^n \int_{B_r^X(x_i)\times B_r^Y(y_i)} c d \pi_i - \sum_{i=1}^n \int_{B_r^X(x_i)\times B_r^Y(y_{i+1})} c d \tilde{\pi}_i\right) \\
& = \int c d\pi - \epsilon \left( \int \sum_{i=1}^n \left(c (U_i(w), V_i(w)) - c(U_i(w),V_{i+1}(w)) \right) d\gamma \right) \\
&< \int c d \pi
\end{align*}
where we have used that $U_i(\spt{\gamma}) \subset B_r^X(x_i)$, and $V_i(\spt{\gamma}) \subset B_r^Y{y_i}$ in the second and third equality and \eqref{eqn: non cyclic ineq} in the last inequality. This contradicts to the assumetion that $\pi$ gives the infimum. Therefore $\pi$ must be $c$-cyclically monotone.
\end{proof}

\begin{Prop}\label{prop: ot stability}
Let $c_n : X \times Y \to \R$ be a sequence of continuous functions that converges uniformly to $c_0$. Let $\mu_n$ and $\nu_n$ be sequences of probability measures on $X$ and $Y$ respectively, and let $\pi_n$ be a Kantorovich solution for Problem \ref{prbm: Kant OT} with marginals $\mu_n$ and $\nu_n$. Suppose that the sequences $\mu_n$ and $\nu_n$ converges weakly to $\mu_0$ and $\nu_0$ respectively. Then, up to a subsequence, $\pi_n$ converges weakly to $\pi_0 \in \Pi(\mu_0,\nu_0)$ which is a Kantorovich solution to the Problem \ref{prbm: Kant OT} with marginals $\mu_0$ and $\nu_0$.
\end{Prop}

$c$-cyclical monotonicity of solutions for the Problem \ref{prbm: Kant OT} demonstrates very important structure of the solutions. If the cost function was the negative inner product in the Euclidean space $c(x,y) = - \langle x, y \rangle$, then $c$-cyclical monotonicity tells that the solution $\pi$ is supported in a monotone set which is given by a graph of superdifferential of a concave function. In general case, the concavity can be replaced with $c$-concavity.

\begin{Def}
A function $\phi : X \to \R$ is called $c$-concave if 
\begin{equation}\label{eqn: c-concave func}
\phi(x) := \inf_{y \in Y}\{ c(x,y) - \psi(y) \}
\end{equation}
for some function $\psi: Y \to \R$. The formula on the right hand side of \eqref{eqn: c-concave func} is called a $c$-transform of $\phi$. For fixed $x_0 \in X$, we define the $c$-superdifferential $\partial^{c}\phi(x_0) \subset Y$ by 
\begin{equation*}
y_0 \in \partial^{c}\phi(x_0) \subset Y \Leftrightarrow \phi(x) \leq c(x,y_0) - c(x_0,y_0) + \phi(x_0)
\end{equation*}
for any $x \in X$. For $A \subset X$, we define $\partial^{c}\phi(A) \subset X \times Y$ by 
\begin{equation*}
\partial^{c}\phi(A) := \{ (x,y) | y \in \partial^{c}\phi(x) \}.
\end{equation*}
\end{Def}

The $c$-cyclical monotonicity of $\pi$ implies that $\spt{\pi}$ lies in a superdifferential of a $c$-concave function. Indeed, if we define
\begin{align*}
 \phi(x) := & \inf_{m \in \N} \inf \left\{ [c (x_1,y_0) - c(x_0,y_0)] + [c(x_2, y_1) - c(x_1,y_1)] + \right.\\
& \left. \cdots + [ c(x,y_m) - c(x_m,y_m)] | (x_i,y_i) \in \spt{\pi}, 1 \leq i \leq m \right\},
\end{align*}
then $\phi$ is a $c$-concave function whose superdifferential $\partial^c \phi(X)$ contains $\spt{\pi}$. For detail, see Rigorous proof of Theorem 5.10, part (i), Step 3 from \cite{villani2008optimal}.

$c$-superdifferential has a nice characterization using $c$-transform. We define $c^*$-transform like \eqref{eqn: c-concave func}.
\begin{equation}\label{eqn: c-transform}
\inf_{x \in X} \{ c(x,y) - \phi(x) \}.
\end{equation}
Then $y \in \partial^{c}\phi(x)$ if and only if 
\begin{equation*}
\phi(x) + \phi^*(y) = c(x,y)
\end{equation*}
where $\phi^*$ is the $c^*$-transform of $\phi$. With this characterization, we can show the following theorem called Kantorovich duality.

\begin{Lem}[Kantorovich Duality]
\begin{equation*}
\inf_{\pi \in \Pi(\mu,\nu)} \int c(x,y) d\pi = \sup_{(\phi,\psi) \in \Phi_{c}} \left( \int \phi d\mu + \int \psi d\nu \right)
\end{equation*}
where 
\begin{equation*}
\Phi_{c} = \{ (\phi,\psi)\in L^1(\mu) \times L^1(\nu)| \phi(x) + \psi(y) \leq c(x,y)\}.
\end{equation*}
\end{Lem}

Kantorovich Duality with characterization of $c$-cyclical monotone sets using $c$-concave functions provides us a very important observation.

\begin{Lem}\label{lem: c-cyclical mono characterize solution}
A measure $\pi \in \Pi(\mu,\nu)$ is a solution to Problem \ref{prbm: Kant OT} if and only if $\pi$ is $c$-cyclically monotone. 
\end{Lem}

\section{Proof of Lemma \ref{Lem: OT ex lem}}\label{sec: proof lem}
\begin{proof}
We define a new coordinate
\begin{equation*}
(u,\ubar) = (x+\xbar, x-\xbar), (v,\vbar) = (y+\ybar, y-\ybar),
\end{equation*}
and define $\Proj{u} (u,\ubar) = u$ and $\Proj{v} (v,\vbar) = v$. Also, we denote
\begin{align*}
\mu_u = u_\sharp (\mu \otimes \mu),\ & \mu_{u,\ubar} = (u,\ubar)_\sharp (\mu \otimes \mu),\\
\nu_v = v_\sharp (\nu \otimes \nu),\ & \nu_{v,\vbar} = (v,\vbar)_\sharp (\nu \otimes \nu).
\end{align*}
Since $\mu$ and $\nu$ are absolutely continuous with respect to the Lebesgue measures, we have
\begin{align*}
\mu_{u,\ubar} = f(u,\ubar) du d\ubar, \ & \mu_u = \left(\int_\R f(u,\ubar) d\ubar\right) du \\ \nu_{v,\vbar} = g(v,\vbar)dvd\vbar, \ & \nu_v = \left(\int_\R g(v,\vbar) d\vbar \right) dv
\end{align*}
Then for any $\pi' \in \Pi(\mu \otimes \mu, \nu \otimes \nu)$, we have
\begin{align*}
\int -(x+\xbar)(y+\ybar) d\pi'(x,\xbar,y,\ybar) & = \int - uv d \pi'_{u,\ubar,v,\vbar}(u,\ubar,v,\vbar)\\
& = \int -uv d (\Proj{u},\Proj{v})_\sharp \pi'_{u,\ubar,v,\vbar}(u,v).
\end{align*}
where $\pi'_{u,\ubar,v,\vbar} = (u,\ubar,v,\vbar)_\sharp \pi' \in \Pi(\mu_{u,\ubar},\nu_{v,\vbar})$. Note that $(\Proj{u},\Proj{v})_\sharp \pi'_{u,\ubar,v,\vbar} \in \Pi(\mu_u,\nu_v)$. We claim that if $\pi' \in \Pi(\mu \otimes \mu, \nu \otimes \nu)$ is a minimizer of Problem \ref{prbm: OT2}, then $(\Proj{u},\Proj{v})_\sharp \pi'_{u,\ubar,v,\vbar}$ is a minimizer of the following problem:
\begin{equation}\label{eqn: new coord OT}
\inf \int -uv d \gamma(u,v), \ \gamma \in \Pi(\mu_u, \nu_v).
\end{equation}
To observe this, we construct a right inverse map of $(\Proj{u},\Proj{v})_\sharp :\Pi(\mu_{u,\ubar}, \nu_{v,\vbar}) \to \Pi(\mu_u,\nu_v)$. For any $\gamma \in \Pi(\mu_u, \nu_v)$, we define $\Phi(\gamma)$ by
\begin{equation*}
\Phi(\gamma) [A\times B \times C \times D] = \int_{A \times C} \frac{\int_B f(u,\ubar)d\ubar}{\int_{\R} f(u,\ubar) d\ubar} \frac{\int_D g(v,\vbar) d\vbar}{\int_{\R} g(v,\vbar)d\vbar}d\gamma(u,v),
\end{equation*}
where $A,B,C,D\subset \R$ are Borel subsets. Then we have
\begin{align*}
\Phi(\gamma)[A\times B \times \R^2] & = \int_{A\times \R} \frac{\int_B f(u,\ubar)d\ubar}{\int_{\R} f(u,\ubar) d\ubar} d\gamma(u,v) \\
& = \int_A \frac{\int_B f(u,\ubar)d\ubar}{\int_{\R} f(u,\ubar) d\ubar} d\mu_u(u) \\
& = \int_A \frac{\int_B f(u,\ubar)d\ubar}{\int_{\R} f(u,\ubar) d\ubar} \left(\int_\R f(u,\ubar) d\ubar\right) du \\
& = \int_{A\times B} f(u,\ubar) d\ubar du \\
& = \mu_{u,\ubar}[A \times B],
\end{align*}
and similarly $\Phi(\gamma)[\R^2 \times C \times D] = \nu_{v,\vbar}[C \times D]$. Therefore we obtain $\Phi(\gamma) \in \Pi(\mu_{u,\ubar}, \nu_{v,\vbar})$. Moreover, we have
\begin{align*}
(\Proj{u},\Proj{v})_\sharp \Phi(\gamma) [A \times C] & = \Phi[(\Proj{u},\Proj{v})^{-1}(A \times C)] \\
& = \Phi[A \times \R \times C \times \R] \\
& = \int_{A \times C} \frac{\int_\R f(u,\ubar)d\ubar}{\int_{\R} f(u,\ubar) d\ubar} \frac{\int_\R g(v,\vbar) d\vbar}{\int_{\R} g(v,\vbar)d\vbar}d\gamma(u,v) \\
& = \int_{A\times C} d\gamma(u,v) \\
& = \gamma[A \times C],
\end{align*}
which shows that $\Phi$ is a right inverse of $(\Proj{u},\Proj{v})_\sharp$. Hence, if $\pi'$ is a minimizer of Problem \ref{prbm: OT2}, we have that for any $\gamma \in \Pi(\mu_u, \nu_v)$,
\begin{align*}
\int -uv d\gamma(u,v) & = \int -uv (\Proj{u},\Proj{v})_\sharp \Phi(\gamma)(u,v) \\
& = \int -uv d \Phi(\gamma)(u,\ubar,v,\vbar) \\
& \geq \int -uv d\pi'_{u,\ubar,v,\vbar}(u,\ubar,v,\vbar) \\
& = \int - uv d(\Proj{u},\Proj{v})_\sharp \pi'_{u,\ubar,v,\vbar}(u,v)
\end{align*}
where we have used that $\pi'_{u,\ubar,v,\vbar}$ is also a minimizer in the corresponding coordinate for the inequality in the third line, and $u = \Proj{u}(u,\ubar)$, $v= \Proj{v}(v,\vbar)$ in the last equality. This shows that $(\Proj{u},\Proj{v})_\sharp \pi'_{u,\ubar,v,\vbar}$ is also a minimizer of \eqref{eqn: new coord OT}. In particular, the support of $(\Proj{u},\Proj{v})_\sharp \pi'_{u,\ubar,v,\vbar}$ is a monotone set as the cost function is $-uv$. On the other hand, as $\mu_u$ and $\nu_v$ are absolutely continuous with respect to the Lebesgue measure, there exists an optimal map $T: \R \to \R$ such that $T$ is injective and $(\Proj{u},\Proj{v})_\sharp \pi'_{u,\ubar,v,\vbar}$ is supported in the graph of $T$ \cite{Figalli2013holder}. The monotonicity of the support together with the injectivity of $T$ implies that if $(k,l) \in \spt{(\Proj{u},\Proj{v})_\sharp \pi'_{u,\ubar,v,\vbar}}$, then
\begin{align*}
\spt{(\Proj{u},\Proj{v})_\sharp \pi'_{u,\ubar,v,\vbar}} \cap \{u \leq k \} & = \spt{(\Proj{u},\Proj{v})_\sharp \pi'_{u,\ubar,v,\vbar}} \cap \{ v \leq l \}, \\
\spt{(\Proj{u},\Proj{v})_\sharp \pi'_{u,\ubar,v,\vbar}} \cap \{u \geq k \} & = \spt{(\Proj{u},\Proj{v})_\sharp \pi'_{u,\ubar,v,\vbar}} \cap \{ v \geq l \}.
\end{align*}
These imply
\begin{align*}
\int_{\overline{D}_k} d\mu \otimes \mu & = \int_{ \{u \leq k \} } d\mu_u 
 = \int_{\{u \leq k \}} d (\Proj{u},\Proj{v})_\sharp \pi'_{u,\ubar,v,\vbar} \\
& = \int_{v \leq l} d (\Proj{u},\Proj{v})_\sharp \pi'_{u,\ubar,v,\vbar} 
 = \int_{v \leq l} d  \nu_v 
 = \int_{\overline{D}_l} d\nu \otimes \nu,
\end{align*}
and similarly $\int_{\underline{D}_k} d \mu \otimes \mu = \int_{\underline{D}_l} d\nu \otimes \nu$.
\end{proof}

\section{Formal derivation of the PDE}
In the first appendix, we derive the Monge-Amp\`ere type equation formally form the quadratic optimal transportation problem with a squared cost function $c(z,\zbar) = \int S(z,w)S(\zbar,w) d\omega(w)$. We assume that $X$ and $Y$ are compact subsets of $\R^n$ and $\mu$ and $\nu$ are absolutely continuous with respect to the Lebesgue measure. We abuse the notation $\mu(x)$ and $\nu(y)$ to denote the density functions of the probability measures $\mu$ and $\nu$. Let $T: X \to Y$ be a solution to the quadratic optimal transportation problem \eqref{prbm: qOT monge}, and let $(f,\phi,\psi)$ be the optimal triple for the dual problem \eqref{eqn: duality}.

From the push-forward condition $T_\sharp \mu = \nu$ and using the change of variables formula, we obtain the Jacobian equation for $T$.
\begin{equation}\label{eqn: jacobian eq}
\nu(T(x))\det(DT) = \mu(x).
\end{equation} 
Since the triple $(f,\phi,\psi)$ is an optimal triple of \eqref{eqn: duality}, $(\phi,\psi)$ is an optimal couple of the following problem.
\begin{equation*}
\sup_{(\phi,\psi)\in \Phi_{S^f}} \left( \int \phi d\mu + \int \psi d\nu \right),
\end{equation*}
which is well-studied from the classical optimal transportation theory. WLOG, we assume that $\phi$ is a $S^f$-convex function. Then \eqref{eqn: sqrt is ot} and Kantorovich duality implies that $T$ is given by the $S^f$-subdifferential of $\phi$. In particular, the following equation holds.
\begin{equation}\label{eqn: sf subdiff}
D\phi(x) = -D_x S^f(x,T(x)).
\end{equation}
Hence, formally we write 
\begin{equation}\label{eqn: T in dphi}
T(x) = E^f(x,D\phi(x))
\end{equation}
where $E^f(x,\cdot)$ is an inverse function of $-D_x S^f(x,\cdot)$. Also, differentiating \eqref{eqn: sf subdiff}, we obtain
\begin{equation*}
D^2 \phi(x) = -D^2_{xx} S^f (x, T(x)) - D^2_{xy} S^f(x, T(x)) DT(x).
\end{equation*}
The above equation with \eqref{eqn: jacobian eq} formally implies the following equation.
\begin{equation}\label{eqn: MA with f}
\det(D^2 \phi(x) + D^2_{xx} S^f (x,T(x))) = \det(-D^2_{xy} S^f(x,T(x))) \frac{\mu(x)}{\nu(T(x))}.
\end{equation}
The next step is to replace $S^f$ with $c$. To achieve this, we recall \ref{rmk: opt f} and use optimality of $T$ and $f$ to obtain
\begin{align*}
f(w) & = \frac{S_\pi (w)}{ \| S_\pi \|_{L^2(\omega)}} = \frac{\int S(z,w) d\pi(z)}{\left( \int \left(\int S(z,w) d\pi(z) \right)^2 d\omega(w) \right)^{1/2}}. \\
& = {\int S(z,w) d\pi(z)}/{\left( \int c(z,\zbar) d\pi(z) d\pi(\zbar) \right)^{1/2}},
\end{align*}
where $\pi = (\id \times T)_\sharp \mu$. Denote $k = (\int c(z,\zbar) d\pi(z) d\pi(\zbar))^{1/2}$, then 
\begin{align*}
S^f(x,y) & = \int S(z,w) f(w) d\omega(w) = \int S(z,w) \frac{\int S(\zbar,w) d\pi(\zbar)}{k}d\omega(w) \\
& = \frac{1}{k} \int \left( \int S(z,w)S(\zbar,w)d\omega(w) \right) d\pi(\zbar) \\
& = \frac{1}{k} \int c(z,\zbar) d \pi(z) \\
& = \frac{1}{k} \int c(x,y,\xbar,T(\xbar)) d\mu(\xbar),
\end{align*}
and
\begin{align*}
D^2_{xx} S^f(x,y) & = D^2_{xx} \frac{1}{k} \int c(x,y,\xbar,T(\xbar)) d\mu(\xbar) \\
& = \frac{1}{k} \int D^2_{xx} c (x,y,\xbar, T(\xbar) ) \mu(\xbar), \\
D^2_{xy} S^f(x,y) & = D^2_{xy} \frac{1}{k} \int c(x,y,\xbar,T(\xbar)) d\mu(\xbar) \\
& = \frac{1}{k} \int D^2_{xy} c (x,y,\xbar, T(\xbar) ) \mu(\xbar).
\end{align*}
Applying this to \eqref{eqn: MA with f}, we obtain 
\begin{align*}
& \det\left(D^2 \phi + \frac{1}{k} \int D^2_{xx} c (x,T(x),\xbar, T(\xbar) ) \mu(\xbar) \right) \\
& = \det\left(-\frac{1}{k} \int D^2_{xy} c (x,T(x),\xbar, T(\xbar) ) \mu(\xbar)\right)\frac{\mu(x)}{\nu(T(x))}. \numberthis\label{eqn: MA}
\end{align*}
Using \eqref{eqn: T in dphi} here, we obtain a partial differential equation about the function $\phi$. Note that $k$ was defined to be a quantity depending on $T$, but one can consider $k$ as a constant that does not depend on $T$ since the optimality of $T$ implies another expression of $k$ that does not depend on $T$. 
\begin{equation*}
k = \inf_{\pi \in \Pi(\mu,\nu)} \left( \iint c d\pi d\pi \right)^{1/2}.
\end{equation*}
\eqref{eqn: MA} is a Monge-Amp\`ere type equation, but it is also non-local. Note that the non-locality of the equation is in fact natural as we have seen in Example \ref{Ex: non local}.

\end{appendices}
\end{document}